\documentclass{amsart}

\usepackage[utf8]{inputenc}
\usepackage{graphicx}
\usepackage{enumerate}
\usepackage{amsmath}
\usepackage{amsthm}
\usepackage{mathrsfs}
\usepackage{mathabx}
\usepackage{breqn}
\usepackage{color}

\usepackage[normalem]{ulem}
\usepackage{etoolbox}
\patchcmd{\thebibliography}{\section*{\refname}}{}{}{}

\newtheorem{theorem}{Theorem}
\newtheorem{lemma}[theorem]{Lemma}

\title{Minimal Dynamical System for $\mathbb{R}^n$}
\author{Ankit Vishnubhotla}

\begin{document}
\setlength{\abovedisplayskip}{5pt}
\setlength{\belowdisplayskip}{5pt}

\begin{abstract}
    We will investigate $\mathbb{R}^n$ as the additive group with the Euclidean topology to give a description of $S(\mathbb{R}^n)$, the phase space of the universal ambit of $\mathbb{R}^n$ and $M(\mathbb{R}^n)$, the phase space of the universal minimal dynamical system, in terms of $M(\mathbb{Z}^n)$, the phase space of universal minimal flow of $\mathbb{Z}^n$. This extends work by Turek for $\mathbb{R}$ to $\mathbb{R}^n$.
    
    \smallskip
    \noindent \newline\textbf{Keywords.} minimal flow, dynamical system, universal ambit
    
    \smallskip
    \noindent \textbf{Classification.} 37B02/5
\end{abstract}

\maketitle

\section{Introduction}
Over the past two decades, there has been a growing interest in topological dynamics related to computing the universal minimal flow (UMF) of a topological group $G$. Much of this has been inspired by seminal work \cite{[KPT]} that shows connections between structural Ramsey theory and topological dynamics. This relationship uncovered that there are examples of extremely amenable topological groups where the universal minimal flow is a singleton, and in other instances where topological groups have non-trivial metrizable UMFs. These computations usually work with infinite-dimensional Polish groups so as to make use of the power of structural Ramsey theory. Some notable results are the extreme amenability of the automorphism group of the countable dense linear order without endpoints $(\mathbb{Q}, <)$ \cite{[P]} and that the space of all linear orderings of a countable set is the universal minimal flow of $S_\infty$, the group of permutations of an infinite set \cite{[GW]}.

Recently, there has been interest in understanding the dynamics of topological groups with UMFs that lack metrizability. Work by Basso and Zucker [Unpublished, see BZ for instance] provides a framework for understanding the dynamics of general topological groups beyond Polish groups, for which metrizability of the UMF typically indicates nice dynamics. They isolate a class of groups that they call Closed Almost Period (CAP) groups, where a topological group $G$ is CAP if the set of almost periodic points (the collection of points that belong to minimal subflows of a $G-$flow) is closed for each $G-$flow. Barto\v sov\' a and Zucker show how this CAP property coincides with metrizability at the level of Polish groups \cite{[BZ]}. Basso-Zucker also establish that CAP groups enjoy strong closure properties and are well behaved. Further in this direction, work by Barto\v sov\' a \cite{[B]} characterizes the UMF of general topological groups that possess certain properties, namely that the topological group $G$ contains a compact normal subgroup $K$ that acts freely on $M(G)$ and that there is a uniformly continuous cross-section $G/K\rightarrow G$.

Despite a great deal of work into understanding the minimal flows and topological dynamics of locally compact groups, work into the universal minimal flow has been much sparser, in part due to the non-metrizability of the UMF \cite{[KPT]}. This follows from a famous theorem of Veech \cite{[V]} which describes that no non-trivial (non-compact) locally compact group can be extremely amenable because every such group admits a free $G$-flow, and so $M(G)$ for $G$ locally compact is typically non-metrizable.  As part of some the work done in the past, Turek studied a prototypical locally compact group, $\mathbb{R}$, and described that its universal minimal flow can be decomposed using the sublattice $\mathbb{Z}$ \cite{[ST]}. More precisely, $M(\mathbb{R}) \cong M(\mathbb{Z}) \times \mathbb{I}/E$, where $\mathbb{I}$ is the unit interval $[0, 1]$ and $E$ is an equivalence relation that equates points on the boundary of $\mathbb{I}$ in a specific manner. This work by Turek expanded on previous work by Balcar and Blaszczyk \cite{[BB]} who proved that the phase space of $M(\mathbb{Z})$ is the absolute of the Cantor cube $\{0, 1\}^{2^\omega}$ using Boolean theoretic methods. This paper will investigate an extension of Turek's work to $\mathbb{R}^n$ and prove that the sublattice $\mathbb{Z}^n$ can be used to achieve a similar decomposition of the universal minimal flow. Future work in this orientation will be focusing on using such subgroup/sublattice techniques to better understand the UMF of $SL_n(\mathbb{R})$ and $GL_n(\mathbb{R})$.\newline

\section{Background}

A dynamical system is a triple $(G, X, \pi)$, where $G$ is a topological $T_0$ group (therefore Tychonoff), $X$ is a compact Hausdorff space, and $\pi$ is a continuous action on X. If we use an additive notation for $G$, where the identity element is 0, then $\pi: G \times X \rightarrow X$ is a continuous map such that:\newline
\indent (1) $\pi(0, x) = x$ for each $x \in X$,\newline
\indent (2) $\pi(g + h, x) = \pi(g, \pi(h, x))$ for each $g, h \in G$ and each $x \in X$.

If $(G, X, \pi)$ is a dynamical system, then the space $X$ is called the phase space of the system $(G, X, \pi)$. We use the notations $\pi^g$ for the homeomorphisms $\pi_g : X \rightarrow X$ and $\pi_x$  for the continuous maps $\pi_x: G \rightarrow X$ defined in the following way: $\pi^g(x) = \pi(g, x)$ and $\pi_x(g) := \pi(g, x)$. The set $\pi_x(G)$ is called the orbit of $x \in X$ in the system $(G, X, \pi)$. If the orbit of a point $x \in X$ is dense in the phase space $X$ of a dynamical system $(G, X, \pi)$ then the tuple $(G, X, \pi; x)$ is called an ambit and the point $x$ is noted as the base point of the ambit $(G, X, \pi; x)$.

Let $(G, X, \pi)$ and $(G, Y, \varrho)$ be dynamical systems and let $\phi: X \rightarrow Y$ be a continuous map. If $\phi \circ \pi^g =
\varrho^g \circ \phi$ for any $g \in G$ then $\phi$ is called a homomorphism of the system $(G, X, \pi)$ into the system $(G, Y, \varrho)$. If we deal with ambits $(G, X, \pi; x)$ and $(G, Y, \varrho; y)$, then a homomorphism of the systems $\phi : (G, X, \pi) \rightarrow (G, Y, \varrho)$ such that $\phi(x) = y$ is called a homomorphism of ambits. In the case when $\phi$ is a homeomorphism (surjection) of spaces then $\phi$ is called an isomorphism (epimorphism) of dynamical systems or ambits.

A dynamical system $(G, X, \pi)$ is called minimal if there is no proper closed non-empty set $M \subseteq X$ such that $\pi^g(M) \subseteq M$ for each $g \in G$. The system is minimal iff the orbit $\pi_x(G)$ is dense in $X$ for each $x \in X$.

An ambit $(G, X, \pi; x)$ is called universal for a group $G$, if for any ambit $(G, Y, \varrho: y)$ there exists a surjective homomorphism of ambits $\phi : (G, X, \pi; x) \rightarrow (G, Y, \varrho: y)$.

    The construction of the universal ambit for any group was presented by Brook in \cite{[RB]}. The phase space of the universal ambit for group $G$ is the Samuel compactification of $G$ with respect to its right uniformity. Equivalently, we can obtain the phase space of this ambit if we take space of all so-called regular ultrafilters with respect to a strong inclusion \cite{[ST]} - $\Subset$ - defined in the following way: if $F$ is closed and $U$ open in $G$ then $F \Subset U$ whenever there exists an open neighborhood $V$ of the identity element such that $V + F \subseteq U$. A family $\mathcal{F}$ of non-empty open subsets of $G$ is called a regular ultrafilter whenever the following conditions hold:\newline
\indent (1) if $F \Subset U$, then either $U \in \mathcal{F}$ or $G \setminus F \in \mathcal{F}$ \newline
\indent (2) for every $U_1, U_2 \in \mathcal{F}$ there is an open, non-empty subset $U \subseteq G$ such that $U \in \mathcal{F}$ and $\text{cl }U \Subset U_1 \cap U_2$.

Let $S(G) = \{\mathcal{F} \subseteq \mathcal{P}(G) : \mathcal{F} \text{ is a regular ultrafilter}\}$. For every open, non-empty subset $U$ of $G$, we set $\tilde{U} = \{\mathcal{F} \in S(G) : U \in \mathcal{F} \}$. The family $\{\tilde{U}: U \text{ is an open non-empty subset of } G\}$ generates a compact, Hausdorff topology on $S(G)$ and the group $G$ can be embedded in $S(G)$ as a dense subspace $\{\mathcal{F}_g: g \in G\}$, where $\mathcal{F}_g = \{U : U \text{ is open in } G \,\&\, g \in U\}$, where the notion of strong inclusion corresponds with a notion of relation of subordination. Let $\pi^G: G \times S(G) \rightarrow S(G)$ be defined in the following way: 
$$\pi^G(g, \mathcal{F}) := L_g(\mathcal{F}),$$where $L_g$ is an extension on $S(G)$ of the left translation $l_g: G \rightarrow G$, expressed by the formula $L_g(\mathcal{F}) = \{g + U : U \in \mathcal{F}\}$.

This construction of the phase space through regular ultrafilters yields the tuple $(G, S(G), \pi_G; 0)$, which is a universal ambit for a group $G$ \cite{[ST]}.

We use the same notions of regular ultrafilters given in the paper. In this case the dynamical system we will start with is the universal ambit for $\mathbb{R}^n$, $(\mathbb{R}^n, S(\mathbb{R}^n), \pi, 0)$. Our ultimate goal is to provide a description for $M(\mathbb{R}^n)$.

\section{Finding the Universal Ambit}
First we describe the space $S(\mathbb{R}^n)$. Here we denote by $\mathbb{R}^n$ the group $\mathbb{R}^n$ with the Euclidean topology. Let $\mathbb{I}$ denote $[0,1]$, the closed interval of $\mathbb{R}$, and finally let $\mathbb{Z}$ be the group of integers.\newline
\indent Now we will describe the integer maps we will use. For any $\vec{z} \in \mathbb{Z}^n$ define shift maps $h_1, \ldots, h_n$ such that $h_1(\vec{z}) = (z_1 + 1, \ldots, z_n), \ldots, h_n(\vec{z}) = (z_1, \ldots, z_n + 1)$. Moreover, for any $N = (v_1, \ldots, v_n) \in \mathbb{Z}^n$ we define $h_N(\vec{z}) = h_1^{v_1} \circ \ldots \circ h_n^{v_n}(\vec{z})$.

For a point $\vec{i} \in \mathbb{I}^n$ we define for a set $A = \{a_1, \ldots, a_j\} \subseteq [n] = \{1, 2, \ldots, n\}$ the map $\psi_A(\vec{i}) = \psi_{a_1} \circ \ldots \circ \psi_{a_j}(\vec{i})$, where
\begin{equation*}
\psi_{a_l}(\vec{i}) = 
\begin{cases}
(i_1, \ldots, i_{a_l-1}, 0, i_{a_l+1}, \ldots, i_n) & i_{a_l} = 1 \\
(i_1, \ldots, i_{a_l-1}, 1, i_{a_l+1}, \ldots, i_n) & i_{a_l} = 0
\end{cases}.
\end{equation*}

\indent Now consider the space $\mathbb{Z}^n \times \mathbb{I}^n$. We are going to define an equivalence relation $h$ such that $\mathbb{R}^n \cong (\mathbb{Z}^n \times \mathbb{I}^n)/h$. In order to do this we must equate coordinates on the boundaries of the interval $\mathbb{I}$ with each other. 

Let $(\vec{z}, \vec{i}) \in (\mathbb{Z}^n \times \mathbb{I}^n)$ be a point in the space we are observing. Suppose k coordinates of $\vec{i}$ take on the value of 1, such that $i_{a_1} = \ldots = i_{a_k} = 1$ and $a_1 < \ldots < a_k$. We can identify the point $(\vec{z}, \vec{i})$ with the points $(h_{N_A}(\vec{z}), \psi_A(\vec{i}))$, where $A \subseteq \{a_1, \ldots, a_k\}$ and $\mathbf{1}_A \in \mathbb{N}^n$ such that $\mathbf{1}_A$ is defined as the indicator set of $A$. 
\[
(\mathbf{1}_A)_j =
\begin{cases}
1 & j \in A \\
0 & \text{otherwise}
\end{cases}.
\]
Thus for every nonzero coordinate $v \in A$ we have that $i_v = 1$. This equivalence relation identifies all points with a 1 value in some set of coordinates of $\vec{i}$ with points that instead step up by 1 in the corresponding coordinate in $\vec{z}$. Moreover, given a point $(\vec{z}, \vec{i})$ where $i_b = 0$, then a point $(\vec{z}', \vec{i}')$, where $i'b = 1$ and $z'_b = z_b - 1$ will be in the same equivalence class as the original point. Since each coordinate operates independently, if a point has 0 value in any set of coordinates of $i$ then it will also be in the same equivalence class as points with 1 value in the same subset of this set of coordinates, and a corresponding step down by 1 in the set/subset of coordinates in $\vec{z}$. This yields a quotient space $\mathbb{Z}^n\times \mathbb{I}^n/h$, which is homeomorphic to $\mathbb{R}^n$. In Figure 1 this is visualized in 2 dimensions, $(\mathbb{Z}^2 \times \mathbb{I}^2)/h$.

\begin{figure}
\centering
\includegraphics[width=0.8\textwidth]{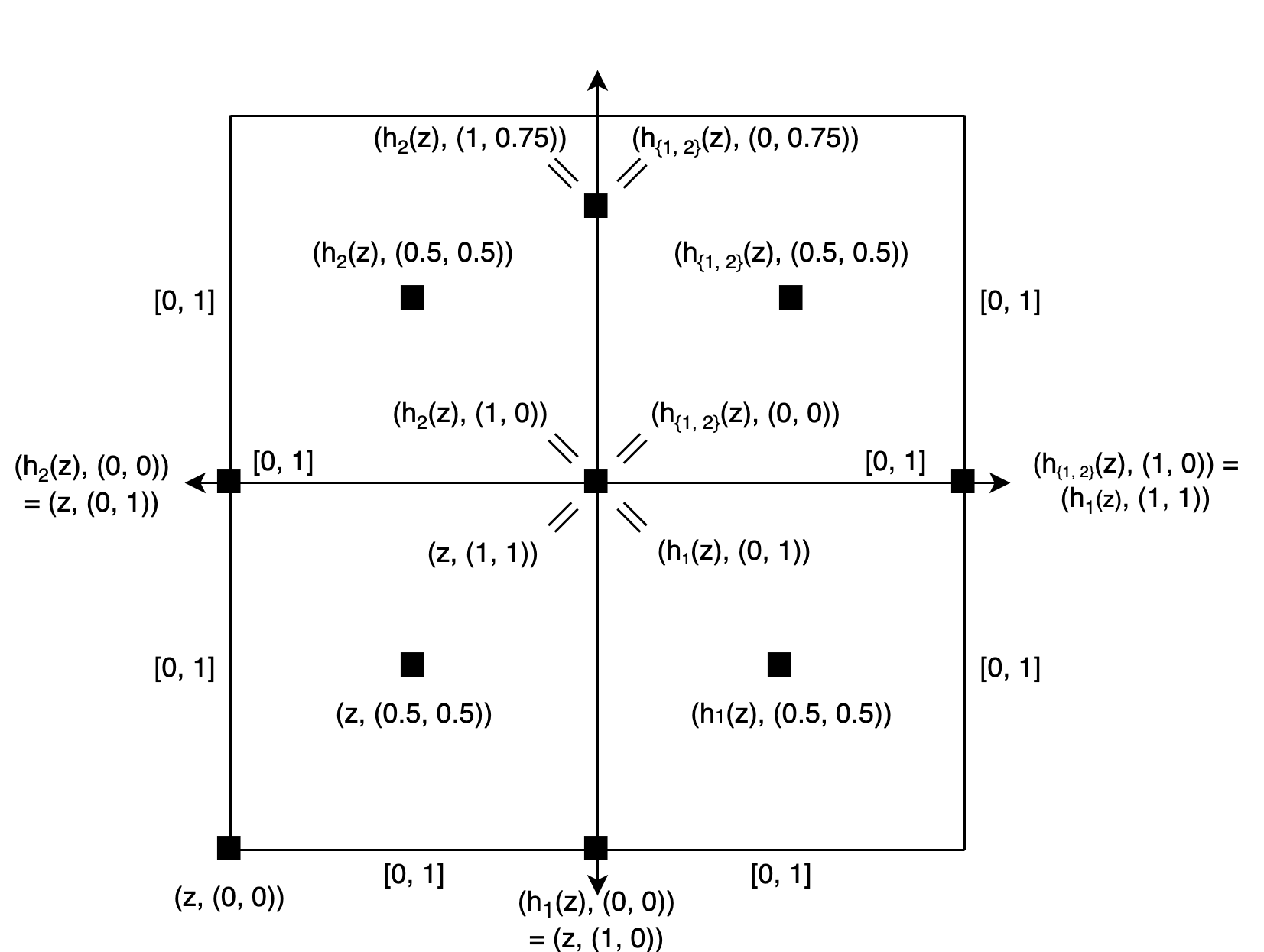}\newline
\caption{subgrid of $(\mathbb{Z}^2 \times \mathbb{I}^2)/h$ with origin in bottom left.}
\end{figure}

Now extend these $h_i$ shift maps for $i = 1, \cdots, n$ to the \v Cech-Stone compactification of $\mathbb{Z}^n$, $\beta(\mathbb{Z}^n)$, such that for a point $p \in \beta(\mathbb{Z}^n)$ we have $h_N(p) = h_1^{v_1} \circ \ldots \circ h_n^{v_n}(p)$. If we identify the point $(p, \vec{i}) \in \beta(\mathbb{Z}^n)\times \mathbb{I}^n$ as equivalent with the points $(h_N(p), \psi_A(\vec{i}))$ for any $A \subseteq \{a_1, \ldots, a_k\}$, where $A$ and $N$ are defined as above, then we obtain a quotient space $\Bigl(\beta(\mathbb{Z}^n)\times \mathbb{I}^n\Bigr)/h$, which is also a compactification of $\mathbb{R}^n$. 
\newline\indent Now for any $\vec{z} \in \mathbb{Z}^n$ and any $\vec{j} \in (0, 1)^n$ define the following homeomorphisms for a point $[(p, \vec{i})]_h \in \Bigl(\beta(\mathbb{Z}^n)\times \mathbb{I}^n\Bigr)/h$: 
\begin{eqnarray*}
\Lambda_1(\vec{z}, [(p, \vec{i})]_h) & = & [(h_{\vec{z}}(p), \vec{i})]_h, \\
\Lambda_2(\vec{j}, [(p, \vec{i})]_h) & = & [(h_N(p), \vec{i}+\vec{j}-N)]_h, 
\end{eqnarray*}where $N \in \mathbb{Z}^n$ is 1 for every coordinate of $\vec{i} +\vec{j} \geq 1$, and $0$ otherwise. Notice that $\mathbf{0} \leq \vec{i} + \vec{j} \leq 2 \cdot \mathbf{1}$, where $\mathbf{0}$ and $\mathbf{1}$ denote the constant vector of all zeros and all ones, respectively.

The first map, $\Lambda_1$ is clearly a homeomorphism since it only moves the point in the $B(\mathbb{Z}^n)$ space. 
The second map, $\Lambda_2$, is also a homeomorphism since any coordinate $k$ of $\vec{i} + \vec{j}$ that falls 
outside of $\mathbb{I}$, meaning that $(i+j)_k \geq 1$, has the corresponding $h_k^{-1}$ map applied, which in turn steps $(i+j)_k$ down by one such that it is $\in \mathbb{I}$. 
Now for any $\vec{x} \in \mathbb{R}^n$ let $\Lambda_{\vec{x}} = \Lambda_1([\vec{x}]) \circ \Lambda_2(\{\vec{x}\}) $, where $\{\vec{x}\} \in \mathbb{R}^n$ is the decimal part of each coordinate of $\vec{x}$ and $[\vec{x}] \in \mathbb{R}^n$ the integer part of each coordinate of $\vec{x}$. Then we can define a map $\sigma(\vec{x}, p) : \mathbb{R}^n \times \Bigl(\beta(\mathbb{Z}^n)\times \mathbb{I}^n\Bigr)/h \rightarrow \Bigl(\beta(\mathbb{Z}^n)\times \mathbb{I}^n\Bigr)/h$ for $\vec{x} \in \mathbb{R}^n$ and $w \in (\beta(\mathbb{Z}^n) \times \mathbb{I})/h$ as 
\[
\sigma(\vec{x}, w) = \Lambda_{\vec{x}}(w).
\]Finally, let $q$ denote the quotient map from $\beta(\mathbb{Z}^n)\times \mathbb{I}^n$ to $\Bigl(\beta(\mathbb{Z}^n)\times \mathbb{I}^n\Bigr)/h$. 
\begin{lemma}
$\biggl(\mathbb{R}^n, \Bigl(\beta(\mathbb{Z}^n)\times \mathbb{I}^n\Bigr)/h, \,\sigma , \, [(0, 0)]_h\biggr)$ is an ambit.
\end{lemma}
\begin{proof}
Clearly $\sigma$ is a group action. We will show that $\sigma$ is continuous. It is sufficient to show that for a neighborhood $V$ of $\sigma(\vec{k}, w)$ that there is a neighborhood $U$ of $(\vec{k}, w)$ such that $\sigma(U) \subset V$. 

We will consider two cases here.

\textit{Case 1}: Let $V$ be a neighborhood of $\sigma(\vec{k}, w)$, where $\vec{k} \in \mathbb{N}^n$, $\vec{i} \in [0, 1]^n$, and $w = [(p, \vec{i})]_h$. Then define $\gamma(\vec{i})^{\delta}_j \in [0, 1]^n$ such that 
\begin{equation*}
\gamma(\vec{i}, \delta)_j = 
\begin{cases}
(1 - \frac{\delta}{2}, 1] & i_j = 1 \\
[0, \frac{\delta}{2}) & i_j = 0 \\
(i_j - \frac{\delta}{2}, i_j + \frac{\delta}{2}) & i_j \in (0, 1)
\end{cases}
\end{equation*}
First, since $w = [(p, \vec{i})]_h$ then there exists a $W \subseteq \beta(\mathbb{Z}^n)$ and $\epsilon_j > 0$ for each $i_j$, such that $(h_{\vec{k}}(p), \vec{i})\in W \times \gamma(\vec{i}, \epsilon) \subseteq q^{-1}(V)$, where $\epsilon = \min\{\epsilon_j: 1 \leq j \leq n\}$. Now let $A = \{a_1, \ldots, a_k\} \subseteq [n]$ be the corresponding coordinate set for $\vec{i}$ such that $i_{a_1} = \ldots = i_{a_k} = 1$ and let $B = \{b_1, \ldots, b_j\} \subseteq [n]$ be the corresponding coordinate set for $\vec{i}$ such that $i_{b_1} = \ldots = i_{b_j} = 0$. Then for any $A' \subseteq A$ and $B' \subseteq B$ take $1_{A'}, 1_{B'} \in \{0, 1\}^n$ as before and we see that, 
\begin{dmath*}
(h_{(\vec{k} + 1_{A'} - 1_{B'})}(p), \psi_{A'+B'}(i)) \in h_{1_{A'} - 1_{B'}}(W) \times \gamma(\psi_{A'+B'}(\vec{i}), \epsilon) \subseteq q^{-1}(V)
\end{dmath*}.

For notation purposes let $$\xi_{A,B} = h_{-\vec{k} + 1_A - 1_B}(W) \times \gamma(\psi_{A+B}(\vec{i}), \epsilon).$$Then here we have 
$$
U = \prod_{j=1}^n\Biggr(k_j -\frac{\epsilon}{2}, k_j + \frac{\epsilon}{2}\Biggr) \times q\Biggl(\bigcup_{A'\in \mathcal{P}(A),\, B' \in \mathcal{P}(B)} \xi_{A', B'}\Biggr),$$
Now $U$ is an open neighborhood of $(\vec{k}, w)$ and $\sigma(U) \subset V$. 

\textit{Case 2}: Let $\vec{k} \in \mathbb{R}^n$, $\vec{i} \in [0, 1]^n$, and $w = [(p, \vec{i})]_h$. 
Continuity of points of the form $\vec{x} \in (0, 1)^n$ is similar. More clearly, the $\Lambda_2$ map is used to shift the proper coordinates of $p$ up or down when $\Lambda_1({\vec{x}}, \vec{i})$ leaves $\mathbb{I}$ in those coordinates.

Now we are considering points of the form $(\vec{k}, w)$, where $\vec{k} \in \mathbb{R}^n \backslash \mathbb{Z}^n$. Let $V$ be an open set such that 
$\sigma(\vec{k}, w) \in V$. Then, using the above results, there are open sets $V_1, V_2$ such that $\{\vec{k}\} \in V_1 \subseteq (0, 1)^n$, $w \in V_2 \subseteq \Bigl(\beta(\mathbb{Z}^n)\times \mathbb{I}^n\Bigr)/h$ and $\sigma(V_1 \times V_2) \subseteq \Lambda_1^{-1}({[\vec{k}]})(V)$. Using this, and that $\sigma$ is the composition of the $\Lambda$ maps, we see that $\sigma(([\vec{x}] + V_1) \times V_2) \subseteq V$. Thus, our $U = ([\vec{x}] + V_1) \times V_2$.
\end{proof}

\begin{theorem}
The Universal Ambit for the group $\mathbb{R}^n$ is isomorphic to the ambit $\bigl(\mathbb{R}^n, (\beta(\mathbb{Z}^n)\times \mathbb{I}^n)/h, \sigma ; [(0, 0)]_h\bigr)$.
\end{theorem}
By the definition of the universal ambit we have the following surjective homomorphism 
\[
\phi : (\mathbb{R}^n, S(\mathbb{R}^n), \pi, 0) \rightarrow (\mathbb{R}^n, \Bigl(\beta(\mathbb{Z}^n)\times \mathbb{I}^n\Bigr)/h, \sigma ; [(0, 0)]_h).
\] By the mapping of base points to each other we have that $\phi(0) = [(0, 0)]_h$, so $\phi \restriction \mathbb{R}^n : \mathbb{R}^n \rightarrow \mathbb{Z}^n \times \mathbb{I}^n/h$ is a map of the form $\phi(\vec{x}) = [([\vec{x}], \{\vec{x}\})]_h$. Since $\phi$ is surjective, it is sufficient to show that it is injective. 

Let $\mathcal{F}, \mathcal{F'} \in S(\mathbb{R}^n)$ such that $\mathcal{F} \neq \mathcal{F'}$. Now take $U \in \mathcal{F}$ and $U' \in \mathcal{F'}$ such that $U \cap U' = \emptyset$. Then we can find open sets $V \in \mathcal{F}$ and $V' \in \mathcal{F'}$ where $\text{cl}_{\mathbb{R}^n} V \Subset U$ and $\text{cl}_{\mathbb{R}^n} V' \Subset U'$, which are disjoint. By the definition of strong inclusion in relation to $\mathbb{R}^n$, there is an $\vec{\epsilon} = (\epsilon_1, \ldots, \epsilon_n)$, where $\epsilon_j > 0$ for $1 \leq j \leq n$ such that $((-\vec{\epsilon}, \vec{\epsilon}) + \text{cl}_{\mathbb{R}^n} V) \cap \text{cl}_{\mathbb{R}^n} V' = \emptyset$. Clearly, $\mathcal{F} \in \text{cl}_{S(\mathbb{R}^n)}\text{cl}_{\mathbb{R}^n} V$ and $\mathcal{F'} \in \text{cl}_{S(\mathbb{R}^n)}\text{cl}_{\mathbb{R}^n} V'$. Now let $F_1 = \phi(\text{cl}_{\mathbb{R}^n} V)$, $F_2 = \phi(\text{cl}_{\mathbb{R}^n} V')$, $K' =\beta(\mathbb{Z}^n)\times \mathbb{I}^n$, and $K = (\beta(\mathbb{Z}^n)\times \mathbb{I}^n)/h$. It suffices to show that $\text{cl}_K F_1 \cap \text{cl}_K F_2 = \emptyset$. 

We will proceed by contradiction. Suppose there is some $[(p, \vec{i})]_h \in \text{cl}_K F_1 \cap \text{cl}_K F_2$. Let $\epsilon = \min\{\epsilon_j: 1 \leq j \leq n\}$ and set $\delta < \frac{\epsilon}{2}$.

\textit{Case 1:} $\vec{i} \in (0, 1)^n$
This means that $(p, \vec{i}) \in \text{cl}_{K'} q^{-1}(F_1) \cap \text{cl}_{K'} q^{-1}(F_2)$. Let 
$$
A_j = \Biggr\{\vec{z} \in \mathbb{Z}^n: \Biggr(\{\vec{z}\} \times \prod_{k=1}^n(i_k - \frac{\delta}{2}, i_k + \frac{\delta}{2})\Biggr) \cap q^{-1}(F_j) \not= \emptyset\Biggr\}, \,\,\,\,j \in \{1, 2\}.$$
We have that $A_1, A_2 \in p$, and by definition of ultrafilters that $A_1 \cap A_2 \in p$. This means there is some $\vec{z} \in A_1 \cap A_2$, and by definition of $A_1, A_2$ there are $[(\vec{z}, i_1)]_h \in F_1$ and $[(\vec{z}, i_2)]_h \in F_2$ where for each coordinate $|(i_{1})_t - (i_{2})_t| < 2\delta < \epsilon$ for $1 \leq t \leq n$. But this contradicts that $((-\vec{\epsilon}, \vec{\epsilon}) + \text{cl}_{\mathbb{R}^n} V) \cap \text{cl}_{\mathbb{R}^n} V' = \emptyset$.

\textit{Case 2:} $\vec{i} \in [0, 1]^n$
\newline Let $A = \{a_1, \ldots, a_l\} \subseteq [n]$ and $B = \{b_1, \ldots, b_m\} \subseteq [n]$ be the coordinate sets for $\vec{i}$ such that $i_{a_1} = \ldots = i_{a_l} = 1$ and $i_{b_1} = \ldots = i_{b_m} = 0$. Now let $C = A \cup -B$ and have $\{C_1, \ldots, C_v\} = \mathcal{P}(C)$. Then define,
\begin{equation*}
(1'_{C_j})_k = 
\begin{cases}
1 & k \in C_j \\
-1 & -k \in C_j \\
0 & \text{otherwise}
\end{cases}
\end{equation*}
Then we have that the preimages are
\[
q^{-1}([(p, \vec{i})]_h) = \{(h_{1'_{C_j}}(p), \psi_{C_j}(\vec{i})),\, C_j \in \mathcal{P}(C)\}.\]
Now since $[(p, \vec{i})]_h \in  \text{cl}_K F_1 \cap \text{cl}_K F_2$ we must have that there are some preimages $x, x'\in q^{-1}([(p, i)]_h)$ such that $x \in \text{cl}_{K'} q^{-1}(F_1)$ and  $x' \in \text{cl}_{K'} q^{-1}(F_2)$. We can represent $x = (h_{1'_{C_{\alpha}}}(p), \psi_{C_{\alpha}}(\vec{i}))$ and $x' = (h_{1'_{C_{\alpha'}}}(p), \psi_{C_{\alpha'}}(\vec{i}))$ for some $1 \leq \alpha, \alpha' \leq v$. Assume $\alpha < \alpha'$ (the proof for $\alpha' < \alpha$ is similar). In this case let $\beta = C_{\alpha'}\setminus C_{\alpha} \cup -(C_{\alpha}\setminus C_{\alpha'})$. Now, using the $\gamma$ map defined before, we see that a set 
\[
X_1 = \Biggr\{\vec{z} \in \mathbb{Z}^n: \Biggr(\{\vec{z}\} \times \gamma^{\delta}(\psi_{C_\alpha}(\vec{i}))\Biggr) \cap q^{-1}(F_1) \not= \emptyset\Biggr\},
\]
belongs to $h_{1'_{C_{\alpha}}}(p)$. Similarly, a set
\[
X_2 = \Biggr\{\vec{z} \in \mathbb{Z}^n: \Biggr(\{\vec{z}\} \times \gamma^{\delta}(\psi_{C_{\alpha'}}(\vec{i}))\Biggr) \cap q^{-1}(F_2) \not= \emptyset\Biggr\},
\]
belongs to $h_{1'_\beta}(h_{1'_{C_{\alpha}}}(p)) = h_{1'_{C_{\alpha'}}}(p)$. 

Since $X_2 \in h_{1'_{C_{\alpha'}}}(p)$ then $X_2 - 1'_\beta \in h_{1'_{C_{\alpha}}}(p)$. Let $z \in X_1 \cap (X_2 - 1'_\beta)$. Then there exist points $[(z, \vec{i_1})]_h \in F_1$ and $[(z+1'_\beta, \vec{i_2})]_h \in F_2$. These points have the property that for coordinate $j \in \beta$ then $(z+1'_\beta)_j = z_j + 1$, $(i_1)_j \in (1-\delta, 1]$ and $(i_2)_j \in [0, \delta)$, or if $-j \in \beta$ then $(z+1'_\beta)_j = z_j - 1$, $(i_1)_j \in [0, \delta)$ and $(i_2)_j \in (1-\delta, 1]$, or lastly if $(z+1'_\beta)_j = z_j$ then $|(i_1)_j - (i_2)_j| < \delta +\delta < \epsilon$. This is a contradiction.
\newline\indent Thus, $\phi$ is injective and is an isomorphism of ambits. 

\section{Universal Minimal Flow for $\mathbb{R}^n$ }
\begin{theorem}
The phase space of the universal minimal dynamical system for the group $\mathbb{R}^n$ is homeomorphic to $M(\mathbb{Z}^n) \times \mathbb{I}^n/h$. 
\end{theorem}
Since the systems $(\mathbb{R}^n, S(\mathbb{R}^n), \pi, 0)$ and $(\mathbb{R}^n, \bigl(\beta(\mathbb{Z}^n)\times \mathbb{I}^n\bigr)/h, \sigma ; [(0, 0)]_h\bigr)$ are isomorphic, the minimal subsets of these systems are also isomorphic. Thus, it suffices to consider a minimal subset in the system $\bigl(\mathbb{R}^n, \bigl(\beta(\mathbb{Z}^n)\times \mathbb{I}^n\bigr)/h, \sigma ; [(0, 0)]_h\bigr)$. Let $M \subseteq \beta(\mathbb{Z}^n)$ be minimal, closed, and invariant in $\beta(\mathbb{Z}^n)$ with $H$. Clearly, $M$ is homeomorphic to $M(\mathbb{Z}^n)$, the universal minimal flow of $\mathbb{Z}^n$. The set $M \times \mathbb{I}^n/h \subseteq \bigl(\beta(\mathbb{Z}^n)\times \mathbb{I}^n\bigr)/h$ is also closed and invariant in the system $\bigl(\mathbb{R}^n, \bigl(\beta(\mathbb{Z}^n)\times \mathbb{I}^n\bigr)/h, \sigma ; [(0, 0)]_h\bigr)$. Lastly, by definition of $M$ being minimal, each $p \in M$ has a dense orbit in $M$ and so any $[(p, i)]_h \in M \times \mathbb{I}^n/h \subseteq \bigl(\beta(\mathbb{Z}^n)\times \mathbb{I}^n\bigr)/h$ also has dense orbit in $M \times \mathbb{I}^n/h$. Thus, it is apparent that $M(\mathbb{R}^n)$ is homeomorphic to $M(\mathbb{Z}^n) \times \mathbb{I}^n/h$.

\section{Open Problems}
Given such a decomposition of a locally compact group, we can ask whether such decompositions also hold in general for the class of locally compact groups. From Barto\v sov\' a we know that if we tighten the requirement for such groups to have a compact normal subgroup that acts freely on the UMF and there exists a uniformly continuous cross-section then we can find a product decomposition of the UMF \cite{[B]}. The decomposition presented in this paper differs in that the second term in the product, $\mathbb{I}$, is not normal and does not allow for a short exact sequence to be formed. More locally compact groups like this can be investigated, such as $GL(n, \mathbb{R})$ and $SL(n, \mathbb{R})$, with differing product decompositions, corresponding equivalence relations to form the space, and group actions. Further work may also seek to generalize the existence of such product decompositions with more requirements on the locally compact group, in the vein of \cite{[B]}.

\renewcommand\refname{\vskip -0.2cm}
\section{References}

\bibliographystyle{plain} 
\bibliography{dcites}
%
%

\end{document}